\documentclass[12pt,fleqn,leqno]{article}
\usepackage{amssymb}

\usepackage[dvips]{graphicx}
\usepackage{bbm}
\usepackage{amsmath}
\usepackage{amsthm}
\usepackage{mathtools}

\allowdisplaybreaks

\numberwithin{equation}{section}

\makeatletter
\renewcommand{\section}{\@startsection{section}{1}{0pt}{20pt}{6pt}{\large\bf}}
\renewcommand{\@seccntformat}[1]{\csname the#1\endcsname.\ }

\def\footnoterule{\kern -3pt \hrule width 2.7 true cm \kern 2.6pt}

\def\ni{\noindent}

\renewcommand{\epsilon}{\varepsilon}

\newcommand{\spint}[4]{\int_{#1}^{#2} \! #3 \, \mathrm{d} #4}
\newcommand{\spsum}[1]{\sum_{\mathclap{#1}}}

\newcommand{\rplusr}[1][]{\mathbb{R}_+ \times \mathbb{R}^{#1}}

\newcommand{\rplusrtor}[1][]{\mathbb{R}_+ \times \mathbb{R}^{#1} \rightarrow \mathbb{R}}

\newcommand{\indic}[1]{\, \mathbbm{1}_{\{ #1 \}}}

\newcommand{\lt}[2]{\ell^{#1}_{#2}}

\newtheorem{theorem}{Theorem}[section]

\theoremstyle{remark}
\newtheorem{remark}[theorem]{Remark}
\theoremstyle{definition}

\theoremstyle{plain}

\allowdisplaybreaks

\begin{document}

\title{\huge \textbf{Change of Variables with Local Time on Surfaces for Jump Processes}}
\author{Daniel Wilson}
\date{}
\maketitle






{\par \leftskip=2.1cm \rightskip=2.1cm \footnotesize
\noindent The `local time on curves' formula of Peskir provides a stochastic change of variables formula for a function whose derivatives may be discontinuous over a time-dependent curve, a setting which occurs often in applications in optimal control and beyond. This formula was further extended to higher dimensions and to include processes with jumps under conditions which may be hard to verify in practice. We build upon the work of Du Toit in weakening the required conditions by allowing semimartingales with jumps. In addition, under vanishing of the sectional first derivative (the so-called `smooth fit' condition), we show that the classical It\^o formula still holds under general conditions.

\par}


\footnote{{\it Mathematics Subject Classification 2010.}
Primary 60H05, 60H30. Secondary 60G40, 60J60.}

\footnote{{\it Key words and phrases:} Change of variables, local time, surfaces, jump process}


\vspace{-20pt}

\section{Introduction} \label{intro}
The prototype non-smooth function which necessitates a generalisation of the classical It\^o formula for $C^2$ functions is the absolute value function, which is linear except at the origin, where it has an irregularity despite remaining continuous. Namely, for a continuous semimartingale $X$, we have Tanaka's formula
\begin{equation}
|X_t| = |X_0| + \spint{0}{t}{\text{sgn}(X_s)}{X_s} + \lt{0}{t},
\end{equation}
where $\lt{0}{t}$ is the local time of $X$. Allowing this irregular point to move along a continuous bounded variation curve $b:\mathbb{R}_+ \to \mathbb{R}$, our problem becomes to find an expansion for the process
\begin{equation}
\left| X_t - b(t) \right|.
\end{equation}
Geometrically the function $F(x,t) = \left| x - b(t) \right|$ is a glueing of two linear functions along $b$, and an expansion is easily obtained by noting that $X-b$ is a semimartingale and so Tanaka's formula applies. More generally, we may try to glue together two smooth functions $F_1$ and $F_2$ along $b$ yielding
\begin{equation} \label{fdef}
F(t,x) = \left\{ \begin{array}{ll}
F_1(t,x) & x \leq b(t), \\
F_2(t,x) & x > b(t),
\end{array} \right.
\end{equation} 
where $F_1 = F_2$ when $x = b(t)$. Finding an expansion for $F(t,X_t)$ falls outside the scope of classical results, but is crucial in applications appearing in mathematical finance and optimal stopping. This motivated the development of the local time on curves formula, a kind of time-dependent It\^o-Tanaka formula due to Peskir \cite{peskir05}, which takes the form
\begin{equation} \label{ltc-formula}
\begin{split}
F(t,X_t) = \, F(0,X_0)  \,  &+  \spint{0}{t}{ F_t(s,X_s -)  }{s} +  \spint{0}{t}{ F_x(s,X_s -) }{X_s} \\&+  \frac{1}{2}\spint{0}{t}{  F_{xx}(s,X_s) \indic{X_s \, \neq \, b(s)} }{\! \left< X,X \right>_s} \\ 
&+  \spint{0}{t}{\,\frac{1}{2} \, \big( F_x(s,X_s \, +) - F_x(s,X_s \, -) \big) \indic{X_s \, = \,\, b(s)}  }{_s \lt{b}{s}(X)},
\end{split}
\end{equation}
where the final integral is with respect to the time variable of the local time $\lt{b}{s}(X)$, defined by $\lt{0}{s}(X - b)$. 

The key motivation for the local time on curves formula appeared in \cite{peskiramerican05}, where it was used to characterise the boundary of the optimal stopping problem associated to the American put option, solving a long-standing open problem. In practice, the expressions for $F_1$, $F_2$ and $b$ are implicit, and one must prove both that $F$ defined by (\ref{fdef}) is continuous and extends to a $C^{1,2}$ function from both sides of the curve $b$, and that $b$ is of bounded variation. Despite being significantly weaker than the global $C^{1,2}$ condition of It\^o's formula, this is difficult or impossible in general. However, the formula itself contains only one-dimensional limits parallel to the space axis of the form
\begin{equation}
F_x(s,b(s) \pm),
\end{equation}
which exist under far weaker regularity conditions. We will follow this idea to obtain weaker conditions under which this formula holds, even for the higher-dimensional analogues over hyper-surfaces and with multidimensional jump processes. Further, if the expression
\begin{equation}
F_x(s,b(s) \, +) - F_x(s,b(s) \, -)
\end{equation}
vanishes, we say that the function $F$ obeys the smooth fit condition over $b$. Once again, this is far weaker than $F$ being $C^{1,2}$ or even $C^1$ at $b$. We will provide conditions in this case which are easily verified, and the resulting formula takes form of the classical It\^o formula, where $b$ may even be of unbounded variation.

This formula has since been applied in many other problems in mathematical finance and optimal stopping; relevant recent examples which required weaker conditions include \cite{gao17,qiu14}, and very recently \cite{detemple18}. Further, in \cite{johnson17}, the authors go to some effort to prove that the (random) curve $b$ is of bounded variation, a condition which we may remove entirely thanks to Theorem \ref{jweaksmooth5} below. One of the first works in the case of jump-diffusion processes \cite{bu18} demonstrates the need for a change-of-variables formula like (\ref{ltc-formula}) for more general processes, especially given contemporary interest in unbounded variation L\'evy models in finance. A formula of the type (\ref{ltc-formula}) was also recently used by \'Etor\'e and Martinez in \cite{etore18} to prove existence and uniqueness for a type of stochastic differential equation involving the local time.

Firstly in Section \ref{previousresults} we will introduce the local time on curves formula, its extension to surfaces, and discuss previous developments. In Section \ref{newprogress} we present two special cases, Theorems \ref{jweaksmooth4} and \ref{jweaksmooth5}, which provide clear and applicable results for common settings in applications. It is hoped this will equip the reader to approach the main Theorems \ref{firsttheorem} and \ref{extendeditoform}, in Sections \ref{ltcsection} and \ref{itoextension} respectively. The main predecessor of this work is the thesis of Du Toit \cite{dutoit}; Theorem \ref{firsttheorem} is a version of this Du Toit's result extended to include jump processes under certain conditions. Theorem \ref{extendeditoform} deals with the smooth-fit case, described below around equation (\ref{smoothfit}), which allows us to remove some technical assumptions when the local time term is no longer present. In this case the result is a strict generalisation of It\^o's formula. Also present are some Remarks \ref{firstremark} -- \ref{jumpremark} and \ref{removelimsrmk}, which further detail how technical barriers may be removed to help apply the formulae.

\section{Previous results} \label{previousresults}

We work throughout with semimartingales whose local time admits a right-continuous modification in the space variable. In this case, the right local time at time $ t \geq 0 $ and level $a \in \mathbb{R}$ of an arbitrary semimartingale $ Y $, denoted $\lt{a}{t}(Y)$, may be defined pathwise almost surely by the limit
\begin{equation} \lt{a}{t}(Y) = \lim_{\epsilon \to 0} \, \frac{1}{\epsilon} \spint{0}{t}{\mathbbm{1}_{ \{ a \, \leq \, Y_s  <  \,a + \epsilon \}  }}{\! \left[ Y,Y \right]^c_s}. \end{equation}
In general, the superscript $c$ denotes the continuous part of a measure. We also mention the symmetric local time, defined similarly by the pathwise almost-sure limit  
\begin{equation} L^a_t(Y) = \lim_{\epsilon \to 0} \, \frac{1}{2\epsilon} \spint{0}{t}{\mathbbm{1}_{ \{ a - \epsilon \,< \, Y_s <\, a + \epsilon \}  }}{\! \left[ Y,Y \right]^c_s}. \end{equation}
There exist discontinuous semimartingales for which the local time does not admit a right-continuous modification in space. In this case, the local time may be defined by the Tanaka formula. For a general overview of local time in the discontinuous case, see \cite{protter04}.

The local time on curves formula was first derived by Peskir in \cite{peskir05}, and later extended to surfaces in \cite{peskir07}. Let us set the scene in the time-space case. First, let $b: \mathbb{R}_+  \to \mathbb{R}$ be a continuous function of bounded variation. Define two sets:
\begin{gather}
 C = \{ (t,x) \in \rplusr \; | \; x < b(t) \}, \\
 D = \{ (t,x) \in \rplusr \; | \; x > b(t) \}.
\end{gather}
The graph of $ b$ is a non-smooth but continuous curve. The sets $C$ and $D$ are the (strict) hypograph and epigraph of this curve respectively. A function $F: \rplusrtor$ is said to satisfy the strong smoothness conditions if it is globally continuous, and:
\begin{gather}
F \text{ is } C^{1,2} \text{ on } \bar{C}, \\
F \text{ is } C^{1,2} \text{ on } \bar{D}.
\end{gather}
This means that the restriction of $F$ to $C$ has a $C^{1,2}$ global extension, and the same for the restriction of $F$ to $D$. Under the strong smoothness conditions for continuous $X$, by \cite{peskir05} we have the local time on curves formula (\ref{ltc-formula}). This is a very strong condition, as even uniform convergence of all derivatives from each side of the curve does not guarantee that the function can be extended to a $C^{1,2}$ function everywhere without further assumptions on the curve $b$. In general terms the geometry of the time-space domain allows us to approach the curve from many directions, forcing us to adopt strong assumptions. In preparation for weakening these assumptions, we say that $F$ obeys the weak smoothness conditions if it is globally continuous and:
\begin{gather}
F \text{ is } C^{1,2} \text{ on } C, \label{wsc1}\\
F \text{ is } C^{1,2} \text{ on } D \label{wsc2}.
\end{gather}

The first formula under weaker conditions was given by Peskir \cite{peskir05}. Assume that $X$ solves the SDE
\begin{equation} \label{xitodiff}
dX_t = \mu(t,X_t) \, dt + \sigma(t,X_t) \, dB_t,
\end{equation}
in It{\^o}'s sense, for $\mu_X$ and $\sigma>0$ continuous and locally bounded, where $B = (B_t)_{t \geq 0}$ is standard Brownian motion. Let $F$ obey the weak smoothness conditions (\ref{wsc1}) and (\ref{wsc2}) above. Then
\begin{equation} \label{ito-diff-form}
\begin{split}
\hspace{-15pt}F(t,X_t) = F(0,X_0) \, &+ \spint{0}{t}{\big( F_t + \mu F_x  + (\sigma^2 / 2) F_{xx}\big)(s,X_s) \, \mathbbm{1}_{\{ X_s \,\neq \,\, b(s)\}}}{s} \\[1ex]
&+ \spint{0}{t}{\big(\sigma F_x\big)(s,X_s) \, \mathbbm{1}_{ \{ X_s \,\neq \,\, b(s) \} }}{B_s} \\[1ex]
&+ \frac{1}{2} \spint{0}{t}{\big( F_x(s,X_s +) - F_x(s,X_s -) \big) \, \mathbbm{1}_{\{ X_s \,= \,\, b(s) \}}}{_s L^b_s(X)},
\end{split}
\end{equation}
if the following three conditions are satisfied:
\begin{gather}
F_t + \mu F_x + (\sigma^2 / 2) F_{xx} \text{ is locally bounded on } C \cup D; \label{con3} \\[0.2ex]
F_x(\cdot , b(\cdot) \pm \epsilon) \to F_x(\cdot, b(\cdot) \pm) \text{ uniformly on } [0,t] \text{ as } \epsilon \downarrow 0 \label{ctscon}; \\
\sup_{0\,< \, \epsilon  \,< \,\delta} V \big( F(\cdot , b(\cdot) \pm \epsilon) \big)(t) < \infty \text{ for some } \delta>0. \label{weirdvarcond}
\end{gather}
Here, $V(G)(t)$ is the variation of $G$ on the interval $[0,t]$. Note that the expression in (\ref{con3}) is the infinitesimal generator of $X$ applied to $F$. This is a natural object, and usually appears in place of the individual derivatives of $F$. Its local boundedness ensures the existence of the time integral in (\ref{ito-diff-form}) and weakens the previous conditions by allowing us to cancel oscillation or divergence of individual derivatives. Also, note that condition (\ref{ctscon}) ensures that the limiting jump in $F_x$ is continuous.

Du Toit \cite{dutoit} has removed the requirement (\ref{weirdvarcond}) by using a different method of proof, and further weakened condition (\ref{ctscon}) to continuity of the map $t \mapsto F_x (t,b(t) +) - F_x(t,b(t) -)$. In fact, he shows this in the more general setting when $b$ also depends on a continuous bounded variation process $A$. 

Let us generalise the setting in the introduction by moving to the discontinuous case, and introducing a process of bounded variation alongside the full semimartingale $X$. Let $(A_t)_{t \geq 0}$ be an adapted process of locally-bounded variation, and let $b:\rplusrtor$ be continuous, such that $(b(t,A_t))_{t \geq 0}$ is a semimartingale. Note now that $A$ and $X$ may have jumps. Define the sets:
\begin{gather} \label{graphsets}
 C = \{ (t,a,x) \in \rplusr[2] \; | \; x < b(t,a) \}, \\
 D = \{ (t,a,x) \in \rplusr[2] \; | \; x > b(t,a) \} \label{graphsets2}.
\end{gather}
We say $F$ satisfies the strong smoothness conditions if it is globally continuous, and:
\begin{gather}
F \text{ is } C^{1,1,2} \text{ on } \bar{C}, \\
F \text{ is } C^{1,1,2} \text{ on } \bar{D}.
\end{gather}
If $F$ is globally continuous and instead we have $C^{1,1,2}$ regularity only on the open sets $C$ and $D$ then we say $F$ obeys the weak smoothness conditions.

By \cite{peskir07}, in the strong smoothness setting we retain the obvious analogue of formula (\ref{ltc-formula}) above, which reads
\begin{equation}
\begin{split} \label{strongsmoothnessltc}
\hspace{-15pt}F(t,&A_t,X_t ) = F(0,A_0,X_0)  \, + \spint{0}{t}{ \,\frac{1}{2} \, \big( F_t(s-,A_{s-},X_{s-}+) + F_t(s-,A_{s-},X_{s-}-) \big) }{s}  \\
 &+ \spint{0}{t}{ \,\frac{1}{2}\, \big( F_a(s-,A_{s-},X_{s-}+) + F_a(s-,A_{s-},X_{s-}-) \big) }{A_s}  \\
&+ \spint{0}{t}{ \,\frac{1}{2}\, \big( F_x(s-,A_{s-},X_{s-}+) + F_x(s-,A_{s-},X_{s-}-) \big) }{X_s}  \\
&+ \frac{1}{2} \spint{0}{t}{ F_{xx}(s-,A_{s-},X_{s-}) \indic{X_{s-} \, \neq \, b_{s-}}  }{\! \left[X,X \right]^c_s}  \\
&+ \spint{0}{t}{\, \frac{1}{2}\, \big( F_x(s-,A_{s-},X_{s-}+) + F_x(s-,A_{s-},X_{s-}-) \big) \indic{X_{s-} \, = \, b_{s-} \, , \, X_s \, = \, b_s} }{_s L^b_s(X)}  \\
&+ \sum_{0 < s \leq t}  \Big(  F(s,A_s,X_s) - F(s-,A_{s-},X_{s-})  \\[-1.2ex]
& \hspace{80pt} - \frac{1}{2} \, \big( F_a(s-,A_{s-},X_{s-}+) + F_a(s-,A_{s-},X_{s-}-) \big) \, \Delta A_s  \\[0.6ex]
& \hspace{80pt} - \frac{1}{2} \, \big( F_x(s-,A_{s-},X_{s-}+) + F_x(s-,A_{s-},X_{s-}-) \big) \, \Delta X_s \, \Big),
\end{split}
\end{equation}
where the final integral is with respect to the time variable of $L^b_s(X)$, the symmetric local time of the semimartingale $(X-b)$ at $0$.

In certain problems in optimal stopping, it is possible to directly verify the `smooth fit' condition, meaning that
\begin{equation} \label{smoothfit}
(t,a) \mapsto  F_x (t,a,b(t,a) +) -  F_x (t,a,b(t,a) -) 
\end{equation}
is identically zero. This indicates that we no longer require the local time component above. It should be noted that the function $F$ may still fail to be $C^{1,1,2}$ at the surface $b$, so the conditions of the classical It\^o formula may not hold.

The results of Du Toit \cite{dutoit} form the basis for the results and proofs in this paper, which are themselves a clever combination of the convolution and linear-interpolation methods of \cite{peskir05}. The extensions in this paper deal with jump processes, provide observations on removing some small technical barriers (see Remarks \ref{firstremark}--\ref{lastremark}), and obtain weaker conditions in the smooth-fit case (Theorem \ref{extendeditoform}). The reader should note that the thesis \cite{dutoit} contains results in the setting of intersecting curves, which this paper does not treat.

\section{Results for diffusion and jump processes} \label{newprogress}

The main results of this paper, Theorems \ref{firsttheorem} and \ref{extendeditoform}, are general but technical. Here we treat some special cases which indicate their utility.

When $X$ satisfies (\ref{xitodiff}), then formula (\ref{ito-diff-form}) holds under conditions (\ref{con3}), (\ref{ctscon}) and (\ref{weirdvarcond}). We now demonstrate an example which shows that the analogous result holds when $A$ and $X$ are discontinuous, under certain conditions on the jumps.

\begin{theorem} \label{jweaksmooth4}
Assume $X$ solves the following SDE
\begin{equation}
\mathrm{d} X_t = \mu_X(t-, A_{t-},X_{t-}) \, \mathrm{d}t + \sigma(t-, A_{t-},X_{t-}) \, \mathrm{d} B_t + \lambda_X(t-, A_{t-},X_{t-}) \, \mathrm{d} Y_t,
\end{equation}
where $Y_t$ is a pure-jump L\'evy process with bounded variation. Assume $A$ satisfies the SDE
\begin{equation}
\mathrm{d} A_t = \mu_A(t-, A_{t-},X_{t-}) \, \mathrm{d}t + \lambda_A(t-, A_{t-},X_{t-}) \, \mathrm{d} Y_t.
\end{equation}
Let $b: \rplusrtor$ and $F: \rplusrtor[2]$ be continuous, with $b$ Lipschitz. Further assume $F=F(t,a,x)$ obeys the weak smoothness conditions (\ref{wsc1}), (\ref{wsc2}).
If the following conditions are satisfied:
\begin{gather}
\begin{split}
 &\text {The function }  (F_t + \mu_X  F_x + \mu_A F_a + (\sigma^2/2) F_{xx})(t,a,x) \\ & \text{is locally bounded;} \label{con8}
\end{split}\\[1ex]
\begin{split}
&\text {The map }  t \mapsto F_x(t,A_t,b_t+) - F_x(t,A_t,b_t-) \text{ is almost surely continuous,} \label{con9}
\end{split} \\[1ex]
\text{We have } \mathbb{P} \left[ X_{s-} = b_{s-} \right] = 0 \text{ for all } 0 < s \leq t.
\end{gather}
then the following change of variable formula holds,
\begin{equation} \label{bigeqn}
\begin{split}
\hspace{-15pt} F(t,A_t,X_t) \; = \;\; & F(0,A_0,X_0) \, + \spint{0}{t}{ \left(\sigma F_x\right) (s-,A_{s-},X_{s-}) \indic{X_{s-} \neq b_{s-}} }{B_s}  \\ 
&+ \spint{0}{t}{ (F_t + \mu_X  F_x + \mu_A F_a + (\sigma^2/2) F_{xx})(s-,A_{s-},X_{s-}) \indic{X_{s-} \neq b_{s-}} }{s} \\
&+ \frac{1}{2} \spint{0}{t}{ \left( F_x(s-,A_{s-},b_{s-}+) - F_x(s-,A_{s-},b_{s-}-) \right)  \indic{X_{s-} = \, b_{s-}}}{_s \lt{b}{s}(X)} \\[0.8ex]
&+ \sum_{0 < s  \leq t} F(s,A_s,X_{s}) - F(s-,A_{s-},X_{s-}),
\end{split}
\end{equation}
for all $ t \geq 0 $, where the final integral is with respect to the time variable of $\lt{b}{s}(X)$, the right-local time of the semimartingale $(X-b)$ at $0$.
\end{theorem}

The requirement that $Y$ has bounded variation ensures that $X$ has bounded variation of jumps, meaning that
\begin{equation}
\sum_{0 < s \leq t} \left| \Delta X_s \right| < \infty,
\end{equation}
for all $t \in \mathbb{R}_+$, which implies that the local time is right-continuous in space (see \cite{protter04}). Further, if $b$ is a Lipschitz function, then the process $b_t = b(t,A_t)$ is of locally-bounded variation whenever $A$ is (see \cite{josephy81}). Finally, the stipulation that $\mathbb{P}\left[ X_{s-} = b_{s-} \right] = 0$ for each $0 < s \leq t$ implies that the presence of the indicator functions in (\ref{bigeqn}) do not change the value of the integrals in which they appear. Their introduction allows us to avoid discussion of the limiting values of the derivatives of $F$ as we approach the surface $b$, which can be hard to check in practice.

In the smooth fit case, we no longer require the local time component above. It should be noted that the function $F$ may still fail to be $C^{1,1,2}$ at the surface $b$, which is also much more difficult to check than the sectional continuity of the derivative. In this case, we may allow an arbitrary surface $b$ and remove conditions on the jumps of $X$. 
 
\begin{theorem} \label{jweaksmooth5}
Assume $X$ satisfies the following SDE
\begin{equation}
\mathrm{d} X_t = \mu_X(t-, A_{t-},X_{t-}) \, \mathrm{d}t + \sigma(t-, A_{t-},X_{t-}) \, \mathrm{d} B_t + \lambda_X(t-, A_{t-},X_{t-}) \, \mathrm{d} Y_t,
\end{equation}
where $Y_t$ is a pure-jump L\'evy process. Let $Z$ be a pure-jump L\'evy process of bounded variation, and assume $A$ satisfies the SDE
\begin{equation}
\mathrm{d} A_t = \mu_A(t-, A_{t-},X_{t-}) \, \mathrm{d}t + \lambda_A(t-, A_{t-},X_{t-}) \, \mathrm{d} Z_t.
\end{equation}
Let $b: \rplusrtor$ and $F: \rplusrtor[2]$ be continuous, such that $F=F(t,a,x)$ obeys the weak smoothness conditions (\ref{wsc1}), (\ref{wsc2}). Assume that:
\begin{gather}
\begin{split}
 &\text {The function }  (F_t + \mu_X  F_x + \mu_A F_a + (\sigma^2/2) F_{xx})(t,a,x) \\ & \text{is locally bounded;} \label{con10}
\end{split}\\[1ex] 
\begin{split}
&\text {The map }  t \mapsto F_x(t,A_t,b_t+) - F_x(t,A_t,b_t-) \text{ is almost surely identically zero.} \label{con11}
\end{split} \\[1ex]
\text{We have } \mathbb{P} \left[ X_{s-} = b_{s-} \right] = 0 \text{ for all } 0 < s \leq t.
\end{gather}
Then the following change of variable formula holds,
\begin{equation}
\begin{split}
\hspace{-15pt} F(t,A_t,X_t) = &F(0,A_0,X_0) \, + \spint{0}{t}{ \left(\sigma F_x\right) (s-,A_{s-},X_{s-}) \indic{X_{s-} \neq b_{s-}} }{B_s}  \\ 
&+ \spint{0}{t}{\left( \lambda_X F_x \right)(s-, A_{s-},X_{s-}-) }{Y_s} \\
&+ \spint{0}{t}{ (F_t + \mu_X  F_x + \mu_A F_a + (\sigma^2/2) F_{xx})(s-,A_{s-},X_{s-}) \indic{X_{s-} \neq b_{s-}} }{s} \\[0.8ex]
&+ \sum_{0 < s \leq t} \Big( F(s,A_s,X_{s}) - F(s-,A_{s-},X_{s-})  - F_x(s-,A_{s-},X_{s-}-) \Delta X_s \Big),
\end{split}
\end{equation}
for all $ t \geq 0 $.
\end{theorem}
Note that the jumps of $X$ are the product of $\lambda$ and the jumps of $Y$, and so the final term above may also be written as
\begin{equation}
\begin{split}
\sum_{0 < s \leq t} \Big( F(s,A_s,X_{s}) - F(s-,A_{s-},X_{s-})  - (\lambda_X F_x)(s-,A_{s-},X_{s-}-) \Delta Y_s \Big).
\end{split}
\end{equation}

\section{Local time on curves for jump processes} \label{ltcsection}

Throughout we fix a filtered probability space $(\Omega,\mathcal{F},(\mathcal{F}_t)_{t \geq 0},\mathbb{P})$. Let $(t,A_t,X_t)_{t\geq0}$ be an $\mathcal{F}_t \,$-semimartingale, where $ (A_t)_{t \geq 0}$ is an $ \mathcal{F}_t $-adapted process of locally-bounded variation. We note that $X$ admits a decomposition
\begin{equation}
X = X_0 + K + M,
\end{equation}
into an $\mathcal{F}_t \,$ local martingale $M$, an $ \mathcal{F}_t $-adapted process $K$ of  locally-bounded variation, and an $ \mathcal{F}_0 $-measurable random variable $ X_0 $, such that $M_0 = K_0 = 0$. Define $b$ and sets $C$ and $D$ as given in (\ref{graphsets}) and (\ref{graphsets2}).

Let $F: \rplusrtor[2]$ be a continuous function such that:
\begin{gather}
F \text{ is } C^{1,1,2} \text{ on } {C} \label{weaksconditions1}, \\
F \text{ is } C^{1,1,2} \text{ on } {D} \label{weaksconditions2}.
\end{gather}
We write $ (b_t)_{t \geq 0}$ to mean the process given by $b_t = b(t,A_t)$ for $t \geq 0 $.

\begin{theorem} \label{firsttheorem}
In the setting given above, assume that the function $F$ obeys the following two criteria:
\begin{gather}
\hspace{-10pt} \text{The limits } F_x (t,a,b(t,a) \pm)\text{ exist for all }(t,a)\in\mathbb{R}_+ \times \mathbb{R}; \label{limitsassumption} \\
\hspace{-10pt} \text{The map } (t,a) \mapsto  F_x (t,a,b(t,a) +) -  F_x (t,a,b(t,a) -) \text{ is jointly continuous on }\mathbb{R}_+ \times \mathbb{R}. \label{joinctsassumption}
\end{gather}
Further assume that the surface $b$ process $A$ and semimartingale $X$ satisfy:
\begin{gather}
\hspace{-10pt}\text{The process } (b_t)_{t \geq 0} \text{ is of locally-bounded variation almost surely;} \\
\hspace{-10pt}\label{jumpass1} \text{We have } \spsum{0 < s \leq t} |\Delta X _s | < \infty \text{ for all } t \geq 0 \text{, almost surely.}
\end{gather}
If there exists a signed measure $\lambda $ on $\mathbb{R}_+$, with locally-finite total variation, and a locally-bounded function $H: \rplusrtor[2]$ such that $H(t,a,x-)$ exists for all $(t,a,x) \in \rplusr[2]$, which satisfy
\begin{equation}
\hspace{-15pt}\begin{split} \label{assumption}
&\spint{0}{t}{F_t(s-,A_{s-},X_{s-} \! - c) \indic{X_{s-}-\, c \, \notin \, (b_{s-} -  \epsilon \,,\, b_{s-} + \,\epsilon] \,}}{s}\\[1ex]
&+ \spint{0}{t}{F_a(s-,A_{s-},X_{s-}  \! - c) \indic{X_{s-}-\, c \, \notin \, (b_{s-} -  \epsilon \,,\, b_{s-} + \,\epsilon]\,}}{A^c_s} \\[1ex]
&+ \spint{0}{t}{F_x(s-,A_{s-},X_{s-}  \! - c) \indic{X_{s-}-\, c \, \notin \, (b_{s-} -  \epsilon \,,\, b_{s-} + \,\epsilon]\,}}{K^c_s}\\[1ex]
& + \frac{1}{2} \spint{0}{t}{F_{xx}(s-,A_{s-},X_{s-}  \! - c) \indic{X_{s-}-\, c \, \notin \, (b_{s-} -  \epsilon \,,\, b_{s-} + \,\epsilon]\,}}{[X,X]^c_s} \\[1ex]
 = &\spint{0}{t}{H(s-,A_{s-},X_{s-} \! - c) \indic{X_{s-}-\, c \, \notin \, (b_{s-} -  \epsilon \,,\, b_{s-} + \,\epsilon]\,}}{\lambda(s)},
\end{split}
\end{equation}
for all $\epsilon > 0$ and all $0 < c \leq \delta$ for some fixed $\delta > 0$, then we have
\begin{equation} \label{ltcformula}
\begin{split}
\hspace{-15pt}F(t,A_t,X_t) = \,\, &F(0,A_0,X_0)  \,+\! \spint{0}{t}{H(s-,A_{s-},X_{s-}-)}{\lambda(s)} + \!\spint{0}{t}{F_x(s-,A_{s-},X_{s-} -)}{M_s} \\
&+ \frac12 \, \spint{0}{t}{\big( F_x(s-,A_{s-},b_{s-}+) - F_x(s-,A_{s-},b_{s-}-) \big)}{_s \lt{b}{s}(X)} \\[0.8ex]
&+ \sum_{0 < s \leq t} \Big( F(s,A_s,X_{s}) - F(s-,A_{s-},X_{s-})  - F_x(s-,A_{s-},X_{s-}-) \Delta M_s \Big),
\end{split}
\end{equation}
for all $t \geq 0$.
\end{theorem}
Proof of this theorem is given in the next section. Note that condition (\ref{assumption}) is the general equivalent of local boundedness of the infinitesimal generator, and hence can be reduced to (\ref{con3}) when $X$ and $A$ solve appropriate SDEs - see Theorems \ref{jweaksmooth4} and \ref{jweaksmooth5} for clarification in the most common cases. More generally, if we have that $\mathrm{d} A^c_s, \, \mathrm{d} K^c_s$ and $\mathrm{d} \left[ X, X \right]^c_s$ all absolutely continuous with respect to Lebesgue measure, then instead it suffices to check that 
\begin{equation}
\begin{split}
&\left( F_t + F_a \frac{\mathrm{d}A^c_s}{\mathrm{d}s} + F_x \frac{\mathrm{d}K^c_s}{\mathrm{d}s} + \frac{1}{2} F_{xx} \frac{\mathrm{d}[X,X]^c_s}{\mathrm{d}s} \right) \indic{x -\, c \, \notin \, (b(s,a) -  \epsilon \,,\, b(s,a)  + \,\epsilon] \,}
\\ & \hspace{20pt}= H \indic{x -\, c \, \notin \, (b(s,a) -  \epsilon \,,\, b(s,a)  + \,\epsilon] \,},
 \end{split}
\end{equation}
for $H,c,\epsilon$ as in the Theorem. This is clearly implied by local boundedness with left limits of
\begin{equation}
F_t + F_a \frac{\mathrm{d}A^c_s}{\mathrm{d}s} + F_x \frac{\mathrm{d}K^c_s}{\mathrm{d}s} + \frac{1}{2} F_{xx} \frac{\mathrm{d}[X,X]^c_s}{\mathrm{d}s},
\end{equation}
the generalised infinitesimal generator. In this case that it is not possible to directly combine these quantities as above, we have the following Remark which implies that the equivalent condition on each quantity alone suffices. Further, the other Remarks give further generalisations and remove technical barriers.
\begin{remark} \label{firstremark}
Considering condition (\ref{assumption}) above, we may generalize to the case where we have finitely many locally-bounded functions $H_1, \dots , H_n  : \rplusrtor[2]$, and signed measures $\lambda_1,\dots,\lambda_n$ on $\mathbb{R}_+$, with locally-finite total variation, such that $H_i(t,a,x-)$ exists for each $(t,a,x) \in \rplusr[2]$, where $i=1,\dots,n$, which satisfy
\begin{equation}
\hspace{-15pt}\begin{split}
&\spint{0}{t}{F_t(s-,A_{s-},X_{s-} \! - c) \indic{X_{s-}-\, c \, \notin \, (b_{s-} -  \epsilon \,,\, b_{s-} + \,\epsilon] \,}}{s}\\[1ex]
&+ \spint{0}{t}{F_a(s-,A_{s-},X_{s-}  \! - c) \indic{X_{s-}-\, c \, \notin \, (b_{s-} -  \epsilon \,,\, b_{s-} + \,\epsilon]\,}}{A^c_s} \\[1ex]
&+ \spint{0}{t}{F_x(s-,A_{s-},X_{s-}  \! - c) \indic{X_{s-}-\, c \, \notin \, (b_{s-} -  \epsilon \,,\, b_{s-} + \,\epsilon]\,}}{K^c_s}\\[1ex]
& + \frac{1}{2} \spint{0}{t}{F_{xx}(s-,A_{s-},X_{s-}  \! - c) \indic{X_{s-}-\, c \, \notin \, (b_{s-} -  \epsilon \,,\, b_{s-} + \,\epsilon]\,}}{[X,X]^c_s} \\[1ex]
 = \,&\sum_{i\,=\,1}^n \, \spint{0}{t}{H_i(s-,A_{s-},X_{s-} \! - c) \indic{X_{s-}-\, c \, \notin \, (b_{s-} -  \epsilon \,,\, b_{s-} + \,\epsilon]\,}}{\lambda_i(s)}, \\[1ex]
\end{split}
\end{equation}
for all $\epsilon > 0$ and all $0 < c \leq \delta$ for some fixed $\delta > 0$. The resultant formula consequently changes in the obvious way.
\end{remark}

\begin{remark} \label{limitremark}
We may also relax condition (\ref{assumption}), requiring only that the left-limits of $H$ exist for all $t$ outside of a $\lambda$-null set, almost surely.
\end{remark} 

\begin{remark} \label{thirdremark}
We may further allow dependence of $H,\,F,\,b$ and $\lambda$ on the underlying probability space, provided they are respectively adapted, predictably measurable processes and a random measure, and obey the conditions of the theorem almost surely.
\end{remark}

\begin{remark} \label{fourthremark}Theorem \ref{firsttheorem} extends in a straightforward manner when we allow dependence on finitely many one-dimensional processes of locally-bounded variation. That is, if we allow one-dimensional processes of locally-bounded variation $ A^1,\dots,A^k $, and a general semimartingale $X$, the surface $b$ is then defined on the domain $\mathbb{R}_+ \times \mathbb{R}^k$ and takes values in $\mathbb{R}$. We may also allow finitely many non-intersecting surfaces, under the obvious extension of the weak smoothness conditions. Note that (\ref{assumption}) also changes correspondingly. 
\end{remark}

\begin{remark} \label{penultimateremark}
In reference to Remark 3.2 of \cite{peskir05}, note that (\ref{limitsassumption}) is implied by convexity or concavity of $F$ on $[b(s,a)-\delta,b(s,a)]$ and $[b(s,a),b(s,a)+\delta]$, for some fixed $ \delta > 0 $, uniformly over $(s,a) \in \rplusr$. 
\end{remark}

\begin{remark} \label{lastremark}
We may replace the use of the right local time by the left or symmetric local time, provided that we replace the left limits in the space variable by the right or symmetric limits respectively. Note especially that the range of the variable $c$ in (\ref{assumption}) must also be changed correspondingly. If we are also using the weaker conditions on $X$ given by Remark \ref{jumpremark}, then these conditions must be replaced by the obvious right or two-sided versions.
\end{remark}

\begin{remark} \label{jumpremark}
The condition (\ref{jumpass1}) is connected to the problem of continuity of the local time process of $ X $. We may generalise by replacing (\ref{jumpass1}) by the following two conditions:
\begin{gather}
\hspace{-10pt} X - b \text{ admits a local time which is right-continuous in space} \text{ at } 0 \text{. That is, } \\[0.6em] \nonumber \lim_{a \downarrow 0} \lt{a}{s} = \lt{0}{s} \, \text{ for all } 0  \leq s  \leq t;\\
\hspace{-10pt} \text{We have } \spsum{0 < s \leq t} |\Delta X _s | \indic{X_{s-} = \, b_{s-}, \, X_s \, >\, b_s} < \infty \text{ for all } t \geq 0 \text{, almost surely.}
\end{gather}
\end{remark}
The existence of such a local time, in the case of L\'evy processes, has been studied by many authors. In particular, see \cite{barlow88} for necessary and sufficient conditions. It is already noted implicitly in the conditions of Theorem \ref{firsttheorem} that processes with so-called `bounded variation of jumps', namely those satisfying (\ref{jumpass1}), admit such a local time. A further special case consists of the $\alpha$-stable L\'evy processes, for $1<\alpha<2$, which also admit such a local time, a result given in \cite{boylan64}.

\section{An extension of It\^o's formula} \label{itoextension}

The local time on curves formula of Section \ref{ltcsection} provides an alternative to It\^o's formula when the function of interest contains a discontinuity in the space derivative over a surface. Formally one can see that, in the smooth-fit case, meaning vanishing of (\ref{smoothfit}), the integral term from (\ref{ltcformula}) which includes the local time vanishes, meaning we may relax the condition that $(b_t)_{t \geq0}$ be of locally-bounded variation. Unless, for example, $b$ is Lipschitz (meaning $(b_t)_{t \geq0}$ is automatically of bounded variation whenever $A$ is), proving that $(b_t)_{t\geq0}$ is of locally-bounded variation can be a difficult task. Relaxing this condition in a more general way shortens many results which employ techniques based on specific cases.

This idea leads to the following theorem, which is now an extension of the classical It\^o formula, not including the local time.
\begin{theorem} \label{extendeditoform}
In the setting of Section \ref{ltcsection}, assume that:
\begin{gather}
\hspace{-10pt}\text{The limits } F_x (t,a,b(t,a) \pm)\text{ exist for all }(t,a)\in\mathbb{R}_+ \times \mathbb{R}; \\
\label{assumption2} \hspace{-10pt}\text{The map } (t,a) \mapsto  F_x (t,a,b(t,a) +) -  F_x (t,a,b(t,a) -) \text{ is identically zero  on } \mathbb{R}_+ \times \mathbb{R}. 
\end{gather}
If there exists a signed measure $\lambda $ on $\mathbb{R}_+$, with locally-finite total variation, and a locally-bounded function $H: \rplusrtor[2]$ such that $H(t,a,x-)$ exists for all $(t,a,x) \in \rplusr[2]$, which satisfy
\begin{equation}
\hspace{-15pt}\begin{split} \label{assumption3}
&\spint{0}{t}{F_t(s-,A_{s-},X_{s-} \! - c) \indic{X_{s-}-\, c \, \notin \, (b_{s-} -  \epsilon \,,\, b_{s-} + \,\epsilon] \,}}{s}\\[1ex]
&+ \spint{0}{t}{F_a(s-,A_{s-},X_{s-}  \! - c) \indic{X_{s-}-\, c \, \notin \, (b_{s-} -  \epsilon \,,\, b_{s-} + \,\epsilon]\,}}{A^c_s} \\[1ex]
&+ \spint{0}{t}{F_x(s-,A_{s-},X_{s-}  \! - c) \indic{X_{s-}-\, c \, \notin \, (b_{s-} -  \epsilon \,,\, b_{s-} + \,\epsilon]\,}}{K^c_s}\\[1ex]
& + \frac{1}{2} \spint{0}{t}{F_{xx}(s-,A_{s-},X_{s-}  \! - c) \indic{X_{s-}-\, c \, \notin \, (b_{s-} -  \epsilon \,,\, b_{s-} + \,\epsilon]\,}}{[X,X]^c_s} \\[1ex]
 = &\spint{0}{t}{H(s-,A_{s-},X_{s-} \! - c) \indic{X_{s-}-\, c \, \notin \, (b_{s-} -  \epsilon \,,\, b_{s-} + \,\epsilon]\,}}{\lambda(s)},
\end{split}
\end{equation}
for all $\epsilon > 0$ and all $0 < c \leq \delta$ for some fixed $\delta > 0$, then we have
\begin{equation} 
\begin{split}
\hspace{-15pt}F(t,A_t,X_t) = \,\, &F(0,A_0,X_0)  \,+\! \spint{0}{t}{H(s-,A_{s-},X_{s-}-)}{\lambda(s)} + \!\spint{0}{t}{F_x(s-,A_{s-},X_{s-} -)}{M_s} \\[1ex]
&+ \sum_{0 < s \leq t} \Big( F(s,A_s,X_{s}) - F(s-,A_{s-},X_{s-})  - F_x(s-,A_{s-},X_{s-}-) \Delta M_s \Big)
\end{split}
\end{equation}
for all $t \geq 0$.
\end{theorem}

Finally, let us note that Remarks \ref{firstremark}\,--\,\ref{penultimateremark} still hold in this new setting, with obvious modifications to the conditions required.

\begin{proof}[Proof of Theorem \ref{extendeditoform}]  \label{proof}

The proof follows the method of Du Toit \cite{dutoit}, making adjustments for the existence of jumps. The method of proof is to approximate the process $t \mapsto b_t = b(t,A_t)$ pathwise, from above and below, by a process of locally-bounded variation. After truncating and smoothing the function $F$, we then take limits to return to the original problem. Smoothing and truncation allows us to apply the Lebesgue-Stieltjes chain rule, and standard It\^o formula, whereupon we can combine derivatives in the form of (\ref{assumption3}). Boundary terms appear from the truncation, which either vanish, or converge to the local-time correction term in  (\ref{ltcformula}).

We assume, through localisation, that the semimartingale $(t,A_t,X_t)_{t\geq 0}$ is bounded, and therefore takes values in a compact set. Let $\rho : \mathbb{R} \rightarrow \mathbb{R}$ be a function which is supported on $[0,1]$, such that $\int_{\mathbb{R}} \rho(y) \, \mathrm{d}y = 1$.

If $t \mapsto b_t$ is of locally-bounded variation, define $\tilde{b}^m = b$ for each $m \in \mathbb{N}$. If not, we define the Moreau envelope of $b$ for each $m \in \mathbb{N}$  by
\begin{equation} \tilde{b}^m(t,a) = \inf_{(s,y) \,\in \,\, \rplusr} \left\{ \,  b(s,y) + \frac{1}{2m} \|(t,a) - (s,y) \|^2 \, \right\} ,\end{equation}
where $\| \cdot \|$ is the usual Euclidean norm. Intuitively, this approximation is the vector sum of the epigraphs of the surface $b$ and $\frac{1}{2m}\| x-y \|^2 $. 

It is known (for example, \cite{rockafellar98}) that the Moreau envelope of a continuous function is Lipschitz, and converges pointwise monotonically to $b$ from below as $ m \to \infty $. By Dini's theorem, this implies uniform convergence on compact sets. 

It has also been shown in the one-dimensional case by Josephy \cite{josephy81} that the composition of a Lipschitz function with a process of locally-bounded variation is again of locally-bounded variation. We observe that this generalises to higher dimensions by considering increments in each variable separately.

Fix $\epsilon > 0$, and $n \in \mathbb{N}$. By uniform convergence, and recalling the localisation above, we can choose $m$ large enough such that $\|\tilde{b}^m - b \| < \epsilon$ on the compact set containing $(t,A_t)$.  We define two convolution approximations to $F$, on either side of the surface $\tilde{b}^m$, by
\begin{equation}
\begin{split}
G^{n,m,\epsilon} (s,a,x) = &\spint{-\infty}{\infty}{F(s,a, x - z / n+ \tilde{b}^m(s,a)) \, \indic{x - z/n > 2\epsilon} \, \rho(z)}{z} 
\\=& \spint{-\infty}{\infty}{ n \,F(s,a , k + \tilde{b}^m(s,a))   \indic{k > 2\epsilon} \, \rho(n(x-k))}{k},
\end{split}
\end{equation}
\begin{equation}
\begin{split}
H^{n,m,\epsilon} (s,a,x) = &\spint{-\infty}{\infty}{F(s,a, x - z / n+ \tilde{b}^m(s,a)) \, \indic{x - z/n \leq -\epsilon} \, \rho(z)}{z} 
\\=& \spint{-\infty}{\infty}{ n \,F(s,a , k + \tilde{b}^m(s,a))   \indic{k \leq -\epsilon} \, \rho(n(x-k))}{k}.
\end{split}
\end{equation}
In anticipation of letting $\epsilon \to 0$, we also define 
\begin{equation}
\begin{split}
G^{n,m} (s,a,x) = &\spint{-\infty}{\infty}{F(s,a, x - z / n+ \tilde{b}^m(s,a)) \, \indic{x - z/n > 0} \, \rho(z)}{z} 
\\=& \spint{-\infty}{\infty}{ n \,F(s,a , k + \tilde{b}^m(s,a))   \indic{k > 0} \, \rho(n(x-k))}{k},
\end{split}
\end{equation}
\begin{equation}
\begin{split}
H^{n,m} (s,a,x) = &\spint{-\infty}{\infty}{F(s,a, x - z / n+ \tilde{b}^m(s,a)) \, \indic{x - z/n \leq 0} \, \rho(z)}{z} 
\\=& \spint{-\infty}{\infty}{ n \,F(s,a , k + \tilde{b}^m(s,a))   \indic{k \leq 0} \, \rho(n(x-k))}{k}.
\end{split}
\end{equation}

We deal only with the $G$ approximations - note that the same arguments apply to the $H$ approximations. It is easily verified that these approximations have $x$ derivatives of all orders, and that:
\begin{align}
&\lim_{\epsilon \to 0} \, G^{n,m,\epsilon}(s,a,x) = G^{n,m} (s,a,x); \\
&\!\lim_{n \to \infty}\, \lim_{m \to \infty} \, G^{n,m} (s,a,x) + H^{n,m} (s,a,x) = F(s,a,x + \tilde{b}^m(s,a)). \label{finallimit}
\end{align}

Fix $t>0$. Define an arbitrary sequence of refining partitions of the interval $[0,t]$ whose mesh tends to zero. Specifically, for each $\tilde{n} \in \mathbb{N}$, define $T^{\tilde{n}} = \{ t_{0}^{\tilde{n}} < \dots < t_{\tilde{n}}^{\tilde{n}} \}$, such that $T^{\tilde{n}} \subset T^{\tilde{n}+1}$, and $\max_{1\leq i \leq \tilde{n}} | t^{\tilde{n}}_i - t^{\tilde{n}}_{i-1} | \rightarrow 0$ as $\tilde{n} \to \infty$.

We deal first with an arbitrary semimartingale $Y$, which will allow us to make an appropriate substitution later. Let us approximate $G^{n,m,\epsilon}$ applied to $(t,A,Y)$ by splitting it into purely stochastic and purely bounded variation increments across the partition $T^{\tilde{n}}$. Consider a single increment of $G^{n,m,\epsilon}$ temporarily fix $\tilde{n} \in \mathbb{N}$, and write $T = T^{\tilde{n}} = \{ 0=t_0 < t_1 \leq \dots < t_{\tilde{n}} = t \}$. Lack of smoothness prevents us from directly applying It\^o's formula. However, we can write this increment as
\begin{align*}
\label{original}
&G^{n,m,\epsilon}(t_i,A_{t_i},Y_{t_i}) - G^{n,m,\epsilon}(t_{i-1},A_{t_{i-1}},Y_{t_{i-1}}) =\\[1ex]
\; \stepcounter{equation}\tag{\theequation}  &G^{n,m,\epsilon}(t_i,A_{t_i},Y_{t_i}) - G^{n,m,\epsilon}(t_{i-1},A_{t_{i-1}},Y_{t_{i}}) \\[1ex]
&+ \; G^{n,m,\epsilon}(t_{i-1},A_{t_{i-1}},Y_{t_{i}}) - G^{n,m,\epsilon}(t_{i-1},A_{t_{i-1}},Y_{t_{i-1}}).
\end{align*}

The observation that the latter (stochastic) increment is adapted allows us to make use of It\^o's formula. We can then apply the usual deterministic calculus pathwise to the non-adapted bounded variation increment. More precisely, apply the extended It\^o formula from \cite[Thm. 18, pg. 278]{protter04}, yielding
\begin{equation} \hspace{-15pt}
\label{rep1}
\begin{split}
& \; G^{n,m,\epsilon}(t_{i-1},A_{t_{i-1}},Y_{t_{i}}) - G^{n,m,\epsilon}(t_{i-1},A_{t_{i-1}},Y_{t_{i-1}}) = \\[1ex]
& \; \spint{t_{i-1}}{t_i}{G^{n,m,\epsilon}_{x} (t_{i-1},A_{t_{i-1}},Y_{s-})}{Y_s} + \frac{1}{2} \spint{t_{i-1}}{t_i}{G^{n,m,\epsilon}_{xx} (t_{i-1},A_{t_{i-1}},Y_{s-})}{[Y,Y]^c_s} \\[1ex]
&+ \, \spsum{t_{i-1} <\, s \, \leq \, t_i} \;\; \, \Big( G^{n,m,\epsilon}(t_{i-1},A_{t_{i-1}},Y_{s}) - G^{n,m,\epsilon}(t_{i-1},A_{t_{i-1}},Y_{s-})  - G^{n,m,\epsilon}_{x} (t_{i-1},A_{t_{i-1}},Y_{s-}) \Delta Y_s \Big).
\end{split}
\end{equation}
To deal with the deterministic increment, we apply the Lebesgue-Stieltjes change of variables pathwise, to obtain
\begin{equation} \label{rep2} \hspace{-15pt}
\begin{split}
& G^{n,m,\epsilon}(t_i,A_{t_i},Y_{t_i}) - G^{n,m,\epsilon}(t_{i-1},A_{t_{i-1}},Y_{t_{i}}) =  \\
&\spint{-\infty}{\infty}{ \Bigg\{
 \spint{t_{i-1}}{t_i}{ F_s(s-,A_{s-}, Y_{t_{i}} - z / n+ \tilde{b}^m_{s-}) }{s} + \spint{t_{i-1}}{t_i}{ F_a(s-,A_{s-}, Y_{t_{i}} - z / n+ \tilde{b}^m_{s-}) }{A^c_s} \\ &+ \spint{t_{i-1}}{t_i}{ F_x(s-,A_{s-}, Y_{t_{i}} - z / n+ \tilde{b}^m_{s-}) }{\tilde{b}^{m,c}_s}  \\[1ex] &+ \spsum{t_{i-1} <\, s \, \leq t_i} \; F(s,A_{s},Y_{t_i} - z/n + \tilde{b}^m_{s})  - F(s-,A_{s-},Y_{t_i} - z/n + \tilde{b}^m_{s-})  \Bigg\}\, \indic{Y_{t_i} - z/n \,>\, 2\epsilon}  \, \rho(z)}{z},
\end{split}
\end{equation}
 where $\mathrm{d}\tilde{b}^{m,c}_s$ is the continuous part of the measure $\mathrm{d} \tilde{b}^m_s$. Note that the presence of the indicator function ensures each non-zero increment is in the area where $F$ is $C^{1,1,2}$. 

With the aim of combining the representations (\ref{rep1}) and (\ref{rep2}), we calculate the $x$ derivatives in (\ref{rep1}) directly. Write $G^{n,m,\epsilon}$ as 
\begin{equation}
\begin{split}
G^{n,m,\epsilon} (s,a,x) = \spint{2 \epsilon}{\infty}{ n \,F(s,a , k + \tilde{b}^m(s,a))  \, \rho(n(x-k))}{k}.
\end{split}
\end{equation}
Now we can differentiate under the integral sign in $x$. Integrating by parts in $k$, noting that $F$ is $C^{1,1,2}$ where required, then changing variables, we obtain
\begin{equation}
\begin{split}
 G^{n,m,\epsilon}_x (s,a,x) = \;&n \, F(s,a , 2 \epsilon + \tilde{b}^m(s,a))  \, \rho(n(x-2 \epsilon)) \\[1em]
 &+ \spint{-\infty}{\infty}{F_x(s,a,x - z/n + \tilde{b}^m(s,a)) \, \indic{x \,- \frac{z}{n} > \, 2 \epsilon} \, \rho(z) }{z} .
\end{split}
\end{equation}
Repeating this procedure gives
\begin{equation}
\begin{split}
  G^{n,m,\epsilon}_{xx} (s,a,x) = \;& n^2 \, F(s,a , 2 \epsilon + \tilde{b}^m(s,a))  \, \rho'(n(x-2 \epsilon)) \\[1ex] &+ n \, F_x(s,a , 2 \epsilon + \tilde{b}^m(s,a))  \, \rho(n(x-2 \epsilon))\\[0.7ex]
 &+ \spint{-\infty}{\infty}{F_{xx}(s,a,x - z/n + \tilde{b}^m(s,a)) \, \indic{x \, -  \frac{z}{n} > \, 2 \epsilon} \, \rho(z) }{z}.
\end{split}
\end{equation}

We may now substitute into the original expression (\ref{original}), making use of the deterministic and stochastic Fubini theorems (see \cite[Thms. 64--65, pg. 210--213]{protter04}), yielding
\begin{align*} \label{bigproof}
 & \hspace{-16pt} G^{n,m,\epsilon}(t_i,A_{t_i},Y_{t_i}) - G^{n,m,\epsilon}(t_{i-1},A_{t_{i-1}},Y_{t_{i-1}}) = \\[1ex]
\hspace{10pt} & \spint{t_{i-1}}{t_i}{n \,F(t_{i-1},A_{t_{i-1}} , 2 \epsilon + \tilde{b}^m_{t_{i-1}})  \, \rho(n(Y_{s-}-2 \epsilon))}{Y_s}  \\[1ex]
&+ \frac{1}{2} \spint{t_{i-1}}{t_i}{n^2 \,F(t_{i-1},A_{t_{i-1}} , 2 \epsilon + \tilde{b}^m_{t_{i-1}})  \, \rho'(n(Y_{s-}-2 \epsilon))\stepcounter{equation}\tag{\theequation} \\[1ex]
&\hspace{50pt} + n\, F_x(t_{i-1},A_{t_{i-1}} , 2 \epsilon + \tilde{b}^m_{t_{i-1}})  \, \rho(n(Y_{s-}-2 \epsilon))}{[Y,Y]^c_s}   \\[1ex]
&+ \spint{-\infty}{\infty}{\Bigg\{\spint{t_{i-1}}{t_i}{F_s(s-,A_{s-},Y_{t_i} - z / n + \tilde{b}^m_{s-}) \indic{Y_{t_i} - \frac{z}{n} > \, 2 \epsilon}}{s} \\[1ex]
 &\hspace{32pt}+ \spint{t_{i-1}}{t_i}{F_a(s-,A_{s-},Y_{t_i} - z / n + \tilde{b}^m_{s-}) \indic{Y_{t_i} - \frac{z}{n} > \, 2 \epsilon}}{A^c_s} \\[1ex]
&\hspace{32pt}+ \spint{t_{i-1}}{t_i}{F_x(s-,A_{s-},Y_{t_i} - z / n + \tilde{b}^m_{s-}) \indic{Y_{t_i} - \frac{z}{n} >\,  2 \epsilon}}{\tilde{b}^{m,c}_s} \\[1ex]
& \hspace{32pt}+ \frac{1}{2} \spint{t_{i -1}}{t_i}{F_{xx}(t_{i-1},A_{t_{i-1}},Y_{s-}\! - z / n + \tilde{b}^m_{t_{i-1}}))\indic{Y_{s-} - \frac{z}{n} > \, 2 \epsilon}}{[Y,Y]^c_s}   \\[0.4ex]
&\hspace{32pt} +\spint{t_{i-1}}{t_i}{F_{x}(t_{i-1},A_{t_{i-1}},Y_{s-}\! - z / n + \tilde{b}^m_{t_{i-1}})\indic{Y_{s-} - \frac{z}{n} > \, 2 \epsilon} }{Y_s} \,  \Bigg\}  \;\, \rho(z)}{z}  \\[1ex]
& \hspace{0pt} + \;\;\spsum{t_{i-1} < s  \leq \, t_i} \; G^{n,m,\epsilon}(t_{i-1},A_{t_{i-1}},Y_{s}) - G^{n,m,\epsilon}(t_{i-1},A_{t_{i-1}},Y_{s-})  - G^{n,m,\epsilon}_{x} (t_{i-1},A_{t_{i-1}},Y_{s-}) \Delta Y_s \\[1ex]
&  \hspace{0pt}+ \;\;\spsum{t_{i-1} < s  \leq \, t_i} \;  G^{n,m,\epsilon}(s,A_{s},Y_{t_i}) - G^{n,m,\epsilon}(s-,A_{s-},Y_{t_i} ) .
\end{align*}

By summing over the partition, we obtain $G^{n,m,\epsilon}(t,A_t,Y_{t}) -  G^{n,m,\epsilon}(0,A_0,Y_0)$. Then re-changing the order of integration, we may use continuity of the convolution approximation to allow the mesh size of the partition to tend to zero. The full expression is
\begin{align*}
 & \hspace{-16pt} G^{n,m,\epsilon}(t,A_{t},Y_{t}) - G^{n,m,\epsilon}(0,A_{0},Y_{0}) = \\[1ex]
\hspace{10pt} & \spint{0}{t}{n \,F(s-,A_{s-} , 2 \epsilon + \tilde{b}^m_{s-})  \, \rho(n(Y_{s-}-2 \epsilon))}{Y_s}  \\[1ex]
&+ \frac{1}{2} \spint{0}{t}{n^2 \,F(s-,A_{s-} , 2 \epsilon + \tilde{b}^m_{s-})  \, \rho'(n(Y_{s-}-2 \epsilon)) \\[1ex]
&\hspace{50pt} + n \,F_x(s-,A_{s-} , 2 \epsilon + \tilde{b}^m_{s-})  \, \rho(n(Y_{s-}-2 \epsilon))}{[Y,Y]^c_s}   \\[1ex]
&+ \spint{-\infty}{\infty}{\Bigg\{\spint{0}{t}{F_s(s-,A_{s-},Y_{s} - z / n + \tilde{b}^m_{s-}) \indic{Y_{s} - \frac{z}{n} > \, 2 \epsilon}}{s} \\[0.4ex]
 &\hspace{32pt}+ \spint{0}{t}{F_a(s-,A_{s-},Y_{s} - z / n + \tilde{b}^m_{s-}) \indic{Y_{s} - \frac{z}{n} > \, 2 \epsilon}}{A^c_s} \stepcounter{equation}\tag{\theequation} \\[1ex]
&\hspace{32pt}+ \spint{0}{t}{F_x(s-,A_{s-},Y_{s} - z / n + \tilde{b}^m_{s-}) \indic{Y_{s} - \frac{z}{n} >\,  2 \epsilon}}{\tilde{b}^{m,c}_s} \\[1ex]
& \hspace{32pt}+ \frac{1}{2} \spint{0}{t}{F_{xx}(s-,A_{s-},Y_{s-}\! - z / n + \tilde{b}^m_{s-}))\indic{Y_{s-}\, - \frac{z}{n} > \, 2 \epsilon}}{[Y,Y]^c_s}   \\[0.4ex]
&\hspace{32pt} +\spint{0}{t}{F_{x}(s-,A_{s-},Y_{s-}\! - z / n + \tilde{b}^m_{s-})\indic{Y_{s-} \,- \frac{z}{n} > \, 2 \epsilon} }{Y_s} \,  \Bigg\}  \;\, \rho(z)}{z}  \\[1ex]
& \hspace{0pt} + \;\;\spsum{0\, < s  \leq \, t} \;\; G^{n,m,\epsilon}(s-,A_{s-},Y_{s}) - G^{n,m,\epsilon}(s-,A_{s-},Y_{s-})  - G^{n,m,\epsilon}_{x} (s-,A_{s-},Y_{s-}) \Delta Y_s \\[0.7ex]
&  \hspace{0pt}+ \;\;\spsum{0 \,< s  \leq \, t} \;\;  G^{n,m,\epsilon}(s,A_{s},Y_{s}) - G^{n,m,\epsilon}(s-,A_{s-},Y_{s} ) .
\end{align*}

We now substitute $Y = X - \tilde{b}^m$. Then, we may replace terms $X_{s}$ by their respective left-limits $X_{s-}$, as the jump set of $X$ is at most countable, and the continuous measures such as $dA^c$ assign zero measure to countable sets. We also change the order of integration freely. Combining the above with the $H$ approximations, we obtain
\begin{align*} \label{bigproofcts}
& \hspace{-16pt} (G^{n,m,\epsilon} + H^{n,m,\epsilon})(t,A_{t},X_{t} - \tilde{b}^m_{t}) - (G^{n,m,\epsilon} +  H^{n,m,\epsilon})(0,A_{0},X_{0} - \tilde{b}^m_{0}) = \\[1ex]
\hspace{10pt}& \spint{0}{t}{n \left[ F(s-,A_{s-} , \zeta + \tilde{b}^m_{s-})  \, \rho(n(X_{s-} - \tilde{b}^m_{s-}-\zeta))\right]_{\zeta = -\epsilon}^{2 \epsilon} }{X_s} \\[1ex]
&+ \spint{0}{t}{n \left[ F(s-,A_{s-} , \zeta + \tilde{b}^m_{s-})  \, \rho(n(X_{s-} - \tilde{b}^m_{s-}-\zeta))\right]_{\zeta = -\epsilon}^{2 \epsilon} }{\tilde{b}^m_s} \\[1ex]
&+ \frac{1}{2} \spint{0}{t}{\,\left[ n^2 F(s-,A_{s-} ,\zeta + \tilde{b}^m_{s-})  \, \rho'(n(X_{s-} - \tilde{b}^m_{s-}-\zeta)) \right]_{\zeta = -\epsilon}^{2 \epsilon} \\[1ex]
\stepcounter{equation}\tag{\theequation} &\hspace{50pt} + \left[ n \, F_x(s-,A_{s-} , \zeta + \tilde{b}^m_{s-})  \, \rho(n(X_{s-} - \tilde{b}^m_{s-}-\zeta))\right]_{\zeta = -\epsilon}^{2 \epsilon}}{[X,X]^c_s}   \\[0.4ex]
&+ \spint{-\infty}{\infty}{\Bigg\{\spint{0}{t}{F_s(s-,A_{s-},X_{s-} - z / n) \indic{X_{s-} - \, \tilde{b}^m_{s-} - \frac{z}{n} \notin (- \epsilon,2 \epsilon]\,}}{s} \\[0.4ex]
&\hspace{32pt}+ \spint{0}{t}{F_a(s-,A_{s-},X_{s-} - z / n) \indic{X_{s-} - \, \tilde{b}^m_{s-} - \frac{z}{n} \notin (- \epsilon,2 \epsilon]\,}}{A^c_s} \\[1ex]
&\hspace{32pt} + \frac{1}{2} \spint{0}{t}{F_{xx}(s-,A_{s-},X_{s-} - z / n) \indic{X_{s-} - \, \tilde{b}^m_{s-} - \frac{z}{n} \notin (- \epsilon,2 \epsilon]\,}}{[X,X]^c_s}   \\[1ex]
& \hspace{32pt}+\spint{0}{t}{F_x(s-,A_{s-},X_{s-} - z / n) \indic{X_{s-} - \, \tilde{b}^m_{s-} - \frac{z}{n} \notin (- \epsilon,2 \epsilon]\,} }{X_s} \, \Bigg\} \;\,  \rho(z)}{z}  \\[1ex]
&+  \hspace{10pt} \spsum{0 < s  \leq t} \;\,\; \spint{-\infty}{\infty}{ \Big\{ F(s,A_{s},X_{s} - z/n) - F(s-,A_{s-},X_{s-} - z/n)  \\[1ex]
& \hspace{90pt} - F_{x} (s-,A_{s-},X_{s-} - z/n) \, \Delta X_s \,\Big\} \, \indic{X_{s-} - \, \tilde{b}^m_{s-} - \frac{z}{n} \notin (- \epsilon,2 \epsilon]\,}  \, \rho(z)}{z} \\[1ex]
&\hspace{55pt} + n \left[ F(s-,A_{s-} , \zeta + \tilde{b}^m_{s-})  \, \rho(n(X_{s-} - \tilde{b}^m_{s-}-\zeta))\right]_{\zeta = -\epsilon}^{2 \epsilon} \, \Delta X_s.
\end{align*}
The square parentheses above represent the difference of the expression they contain, evaluated at the subscript and superscript limits.

Now we may pass to the limit as $\epsilon \to 0$. Decomposing $X = K + M$, we employ the assumption (\ref{assumption3}). Letting $\epsilon \to 0$, we may then use  (\ref{assumption2}) and continuity of $ F$. We are left with
\begin{align*}
& (G^{n,m} + H^{n,m})(t,A_{t},X_{t} - \tilde{b}^m_{t}) - (G^{n,m} +  H^{n,m})(0,A_{0},X_{0} - \tilde{b}^m_{0}) = \\[1ex]
& \spint{0}{t}{ \, \Bigg\{ \spint{-\infty}{\infty}{ H(s-,A_{s-},X_{s-}-z/n) \indic{X_{s-}- \frac{z}{n} \neq \, \tilde{b}^m_{s-}} \; \rho(z)}{z} \Bigg\} \; }{\lambda(s)}    \\[1ex]
\stepcounter{equation}\tag{\theequation} &+ \spint{0}{t}{ \, \Bigg\{ \spint{-\infty}{\infty}{ F_x(s-,A_{s-},X_{s-}-z/n) \indic{X_{s-}- \frac{z}{n} \neq \,\tilde{b}^m_{s-}} \; \rho(z)}{z} \Bigg\}  }{M_s} \\[1ex]
&+ \; \spsum{0 <\, s \, \leq t} \;\,\; \spint{-\infty}{\infty}{ \, \Bigg\{ F(s,A_{s},X_{s} - z/n) - F(s-,A_{s-},X_{s-} - z/n)  \\[1ex]
& \hspace{70pt} - F_{x} (s-,A_{s-},X_{s-} - z/n) \, \Delta M_s \,\Bigg\} \, \indic{X_{s-}- \frac{z}{n} \neq \,\tilde{b}^m_{s-}}\; \rho(z)}{z}.
\end{align*}

The indicator function may be removed as $ \{n(X_{s-} - b_{s-}^m )\} $ has zero Lebesgue measure in the $z$ variable. Using (\ref{finallimit}) and the dominated convergence theorem, we may now take the limit as $m \to \infty$, then let $n \to \infty $, to give the result.
\end{proof}

\begin{proof}[Proof of Theorem \ref{firsttheorem}]
The proof of the local time on curves formula for jump processes, Theorem \ref{firsttheorem}, follows in the same way as the proof of the extended It\^o formula, Theorem \ref{extendeditoform}, with the following considerations. Recall that the approximation $\tilde{b}^m$ is simply $b$ in this setting.

Note that the absence of condition (\ref{assumption2}) means that the first-derivative boundary term associated to $\mathrm{d}[X,X]^c_s$ in (\ref{bigproofcts}),
\begin{equation}
\frac{1}{2} \spint{0}{t}{ \, \left[ n \, F_x(s-,A_{s-} , \zeta + \tilde{b}^m_{s-})  \, \rho(n(X_{s-} - \tilde{b}^m_{s-}-\zeta))\right]_{\zeta = -\epsilon}^{2 \epsilon}}{[X,X]^c_s}\, ,
\end{equation}
does not vanish as $\epsilon \to 0$. To introduce the local time, we prove weak convergence of the measures $d_t J^n_t$, given below, to the measure $ d_t \lt{0}{t}(X-b) $ associated to the one-sided local time of $X - b$ at zero, using the occupation time formula. Define $J^n$ by
\begin{equation}
\begin{split}
J^n_t = &\spint{0}{t}{ n\, \rho(n(X_{s-}\! - b_{s-})) }{ [X - b, X - b]^c_s}.
\end{split}
\end{equation}
Note that the measure $ \mathrm{d} [X - b, X - b]^c_s$ assigns zero measure to countable sets, and as $X$ and $b$ are c\'adl\'ag processes, they have at most countably many jumps. Thus we may replace $X_{s-} \! - b_{s-}$ by $X_s - b_s$ above. 

We then have, by the occupation time formula from \cite{protter04} (Corollary 1, pg. 219), that
\begin{equation}
J^n_t = \spint{-\infty}{\infty}{n\, \rho(na)\, \lt{a}{t} }{ a} = \spint{-\infty}{\infty}{\lt{a/n}{t} \, \rho(a) }{a}.
\end{equation}
Under (\ref{jumpass1}), or more generally assuming that the conditions of Remark \ref{jumpremark} hold, by \cite{protter04} (Theorem 76, pg. 228), we may take a version of $\lt{a}{t}$ which is jointly right-continuous in $a$ and continuous in $t$, at the level $0$ in space.  It then follows that, as $n \rightarrow \infty$, we have $J^n_t \rightarrow \lt{0}{t}$ for each $t \geq 0$. This implies weak convergence of the measures $d_t J^n_t$ to $d_t \lt{0}{t}$. In fact, $d_t J^n_t ([a,b))$ converges to $d_t \lt{0}{t} ( [a,b) )$ for all such half-open intervals.

For fixed $\omega \in \Omega$, define the function 
\begin{equation}
s \mapsto g(s) = F_x(s,A_{s},b_{s}+) -  F_x(s,A_{s},b_{s}-). 
\end{equation}

We know that $g$ is right-continuous with left limits (for almost-all such $\omega$) by (\ref{joinctsassumption}). By the standard theory of regulated functions, $g$ admits a uniform approximation by right-continuous step functions. That is, for fixed $\varepsilon > 0$ there exists $h$ such that $\| h-g \|_\infty < \varepsilon$ on $[0,t]$, with
\begin{equation}
h(s) = \sum_{i=0}^N h_i \,\mathbbm{1}_{[a_i,a_{i+1})} (s) + h_{N+1} \mathbbm{1}_{t}(s),
\end{equation}
where $N \in \mathbb{N}$, the $h_i$ are real-valued constants, and $ 0 = a_1 < a_2 < \dots < a_N = t$. We employ again the time-dependent occupation time formula, yielding
\begin{equation}
\begin{split}
&\spint{0}{t}{\ g(s) \, n \, \rho(n(X_{s-}\! -b_{s-})) }{[X-b,X-b]^c_s} = \\[1ex]  &\spint{-\infty}{\infty}{ \left( \spint{0}{t}{ g(s) }{_s \lt{a}{s}} \right) n \, \rho(na)}{a} 
\; = \spint{-\infty}{\infty}{ \left( \spint{0}{t}{ g(s)  }{_s \lt{c/n}{s} } \right) \rho(c)}{c}.
\end{split}
\end{equation}
This gives us the estimate
\begin{equation}
\begin{split}
\left| \spint{0}{t}{ g(s)  }{_s \lt{c/n}{s} } - \spint{0}{t}{  g(s)  }{_s \lt{\, 0}{s} } \, \right| &\leq \left| \spint{0}{t}{ \left( g(s) - h(s) \right) }{_s \lt{c/n}{s} } - \spint{0}{t}{ \left( g(s) - h(s) \right) }{_s \lt{\, 0}{s} } \, \right| \\[1ex] &+ \left| \spint{0}{t}{ h(s)  }{_s \lt{c/n}{s} } - \spint{0}{t}{  h(s)  }{_s \lt{\, 0}{s} } \, \right| \! .
\end{split}
\end{equation}

For any $\delta>0$, we can take $M$ large enough such that the final term is bounded above by $\delta$, for all $0 < c/n < 1/M$. This follows from the convergence of $d_s \lt{c/n}{s}$ to $d_s \lt{\, 0}{s}$ for each indicator function in $h$. The first term is bounded by 
\begin{equation}
2 \varepsilon \sup_{0\,<\,c\,<\,1/M}  \lt{c}{t} .
\end{equation}
Each of these terms can be made arbitrarily small. So we have shown 
\begin{equation}
\begin{split}
 \lim_{n \to \infty} \; \frac{1}{2} \, &\spint{0}{t}{\left( F_x(s-,A_{s-},b_{s-} +) - F_x(s-,A_{s-},b_{s-}-) \right) \, n \, \rho(n(X_{s-}\!\! -b_{s-})) }{[X-b,X-b]^c_s} \\[0.8ex]
= \frac{1}{2} &\spint{0}{t}{\left( F_x(s-,A_{s-},b_{s-} +) - F_x(s-,A_{s-},b_{s-}-) \right) }{_s \lt{ 0}{s}(X-b)},
\end{split}
\end{equation}
where convergence holds almost surely.
Finally, when letting $n \to \infty$, we must employ (\ref{jumpass1}), or the corresponding condition of Remark \ref{jumpremark}, to ensure that the limiting jump-terms form an absolutely convergent sum. This completes the proof of Theorem \ref{firsttheorem}.
\end{proof}

The following remark establishes the `local time on surfaces' formula under strong smoothness conditions when the process obeys (\ref{jumpass1}).
\begin{remark}
Observe that the conditions on the function $F$ in the `local time on surfaces' formula under strong smoothness conditions, equation (\ref{strongsmoothnessltc}), immediately give us that the function $(t,a) \mapsto F_x(t,a,b(t,a)+) - F_x(t,a,b(t,a)-)$ is jointly continuous, and $(b_t)_{t\geq0}$ is a semimartingale by assumption.

Note that strong smoothness immediately gives local boundedness of $F_t \,,F_a$ and $F_{xx}$, and ensures that their left limits exist everywhere. Combining Remarks \ref{firstremark} and \ref{lastremark}, we may immediately see that the left-hand side of (\ref{assumption}) is already of the required form. We assume (\ref{jumpass1}) holds, giving the result.
\end{remark}

\begin{remark}
Consider equation (\ref{ito-diff-form}). In this case, there is no bounded variation process $A$. Note that $t \mapsto F_x(t,b(t) \pm \epsilon)$, are continuous functions in $t$ for each fixed $\epsilon>0$, and converge uniformly. This ensures that the map $t \mapsto F_x(t,b(t)+) - F_x(t,b(t)-)$ is continuous.

Take $ \lambda $ to be the Lebesgue measure. Note that by the occupation times formula the set $\{s \in [0,t] \, | \, X_s = b(s) \}$ is $\lambda$-null, almost surely. Outside this set, the infinitesimal generator appearing in (\ref{con3}) is continuous, and so has left limits. Employing Remark \ref{limitremark}, and taking $H$ to be the expression in (\ref{con3}), the theorem follows.
\end{remark}

We note next an important consideration regarding limits in the integral terms, which is presented in \cite{peskir05} (Remark 2.4, pg. 15). 
\begin{remark} \label{removelimsrmk}
If 
\begin{equation} \label{probzero}
\mathbb{P}(X_{s-} = b_{s-}) = 0,
\end{equation}
for all $ 0 < s \leq t$, then we find that
\begin{equation}
\spint{0}{t}{\indic{X_{s-} =\, b_{s-}}}{s} = 0
\end{equation}
almost surely, and so we may introduce this indicator function into the time integral, eliminating the limits in the space variable. In fact, if $X$ solves an SDE such as (\ref{xitodiff}) then (\ref{probzero}) is satisfied, and furthermore
\begin{equation}
\spint{0}{t}{\indic{X_{s-} = \, b_{s-}}}{X_s} = 0.
\end{equation}
This may be shown using the extended occupation times formula and Fubini theorem. Hence we may eliminate limits in the stochastic integral in the same way. Similar considerations apply to other integral terms.
\end{remark}

\bibliographystyle{napalike}
\bibliography{bib}{}


\par \leftskip=24pt

\ni Daniel Wilson \\
School of Mathematics \\
The University of Manchester \\
Oxford Road \\
Manchester M13 9PL \\
United Kingdom \\
\texttt{daniel.wilson-2@manchester.ac.uk}

\par

\end{document}